\documentclass[10pt]{amsart}

\usepackage{amsmath,amssymb,amscd,amsfonts,verbatim}
\usepackage{enumerate}
\usepackage[all]{xy}

\newtheorem{teorema}{Theorem}[section]

\newtheorem{lema}[teorema]{Lemma}
\newtheorem{proposicion}[teorema]{Proposition}

\newtheorem{remark}[teorema]{Remark}
\newtheorem{notation}[teorema]{Notation}

\def\O{\mathcal O}

\title [Five singular fibers]{ Towards the classification of semistable fibrations having exactly five singular fibers}

\author{Margarita Casta\~ neda-Salazar,  Margarida Mendes Lopes\textsuperscript{1} and Alexis  Zamora }

\begin{document}

\begin{abstract}

 Let $X$ be a nonsingular complex projective surface. Given a semistable non isotrivial fibration $f: X \to \mathbb P^1$  with general non-hyperelliptic fiber of genus $g\geq 4$, we show that, if the number  of singular fibers is 5, then   $g\leq 11$, thus improving the previously known bound $g\leq 17$.   Furthermore,  we show that, for each possible genus, the general fiber has gonality at most 5. The  corresponding fibrations are described  as the resolution of concrete pencils of curves on minimal rational surfaces.
\medskip


 
 \end{abstract}

\maketitle

  \footnotetext[1]{Partially supported by  FCT/Portugal  through Centro de An\'alise Matem\'atica, Geometria e Sistemas Din\^amicos (CAMGSD), IST-ID, projects UIDB/04459/2020 and UIDP/04459/2020.}

 \setcounter{tocdepth}{1}

\section{\bf Introduction}

Let $X$ be a nonsingular complex projective surface and $f:X\to B$  a fibration over a curve,  i.e., $B$ is a nonsingular projective   curve, and $f$ is an epimorphism with connected fibers.  

The  fibration $f$ is said to be  {\sl non isotrivial}  if two general fibers are not isomorphic, i.e., if  the induced map  from $B$ to the moduli space $ \mathcal{M}_g$ of  nonsingular projective   curves of genus $g$  is not constant. 

The  fibration $f$ is said to be {\sl semistable} if all the fibers of $f$ are semistable in the sense of Mumford-Deligne,  i.e., every fiber is reduced with at most nodes as singularities and no fiber contains a $(-1)$-curve. The fibration   is said to be {\sl stable} if is semistable and no fiber contains a $(-2)$-curve and it is {\sl strongly stable} if, moreover,  no fiber contains a $(-3)$-curve. 

\medskip

The problem of determining the minimal number $s$ of singular fibers of a semistable non  isotrovial fibration, when  $B\simeq \mathbb P^1$,  has been studied for a long time, starting with the work of  Beauville   (\cite{beau-82}), and nowadays we know the following: 

\begin{itemize}
 \item [i)] $4\le s$ and if $s=4$, then $g=1$ (\cite{beau-82 A}, \cite{tan}),
 \item[ii)] if $X$ has non negative Kodaira dimension, then $6\le s$ (\cite{ttz}),
\item[ iii)] if $X$ is of general type, then $7\le s$ (see \cite{ltz}). 
\end{itemize}

\medskip

Having these results, a natural problem is characterizing fibrations with the minimal possible number of singular fibers. 
In \cite{beau-82}, Beauville gave a classification of semistable fibrations with $4$ singular fibers and general fiber of genus 1, that, in view of latter results of Tan  in \cite{tan}, proved to cover all possibilities for $s=4$.

\medskip 

   By \cite{ttz},  the case $s=5$ can only occur if the surface is rational or birationally ruled.  In \cite{C-Z}, the first and third authors of the present paper have shown that, if $s=5$, then  $K_X^2= 4-4g$ (i.e.,  $(K_X+F)^2=0$), unless $X$ is rational and $g\leq 17$.   So, in particular, if $q>0$  then $K_X^2=4-4g$.  Note  that surfaces with $K_X^2=4-4g$ are characterized in   Theorem 2.1 of \cite{ttz}.  More recently, in \cite{ConjectureTan},  some restrictions on the possible  irregularities $q(X)$  occurring  for $s=5$ were established.  
  
  For rational surfaces, if $s=5$ and the general fiber of $f$ is hyperelliptic, then also $K_X^2=4-4g$ (see  Theorem \ref{g17} of the present paper). 

\medskip

In the present paper  we prove the following:

\begin{teorema}\label{Classification} Let $X$ be a  rational nonsingular  complex projective surface  and  $f: X \to \mathbb{P}^1$ be a semistable non isotrivial  fibration of genus $g\ge 4$. If the number $s$ of singular fibers is equal to $5$ and the general fiber $F$ of $f$ is non-hyperelliptic, then    one of the following occurs:

\begin{itemize}

\item[i)]  $K_X^2=2-3g$,  $g\leq 11$  and the general fiber $F$ is trigonal.  If $n$ denotes its Maroni's invariant, then the embedding of  $F$  in 
$\mathbb{F}_n$  has class $3\Delta+(\frac{g+n}{2}+n+1)\Gamma$  where $g+2\geq 3n$,  $\Gamma$ is the general fiber of the structural morphism and $\Delta$ is the section satisfying $\Delta^2=-n$. The fibration $f$ is obtained by blowing up the base locus  of a pencil  $\Lambda$  with general nonsingular member  of curves  in $|3\Delta+(\frac{g+n}{2}+n+1)\Gamma|$; 

  \item [ii)] $K_X^2=-16$, $g=6$ and $f$ is obtained by blowing up the base locus  of a   pencil $\Lambda$  with general nonsingular member   of plane  degree $5$ curves; 
    \item [iii)] $K_X^2= -10$, $g=4$ and $f$ is obtained by blowing up the base locus  of a pencil  $\Lambda$ with general nonsingular member of cubic hypersurface sections   of the quadric cone in $\mathbb P^3$;

\item[iv)] $K_X^2=3-3g$, $g\leq 10$, and $f$ is obtained by blowing up the base locus of a  pencil  $\Lambda$  of plane degree $6$ curves, whose general element admits only  $10-g$  singularities of order $2$;

\item[v)] $K_X^2=-24$, $g=9$ and $f$ is obtained by blowing up the base locus of a pencil  $\Lambda$ with general nonsingular member  of quartic hypersurface sections  of a  nonsingular quadric in $\mathbb P^3$. 

  \end{itemize}
For all  the cases, except iv),   the base points of $\Lambda$ are  simple base points (possibly infinitely near).  In case iv),  $\Lambda$  has also simple base points (possibly infinitely near), apart  from the  $10-g$ double base points.
 
 The base points of $\Lambda$ are not in  a general position, except possibly  in the cases $g=4$ or $g=5$, when $K_X^2=3-3g$ (case iv)). 

 In addition,  in cases i) and iii), the fibration $f$ is not strongly stable and  in case i),  if $g\geq 7$, $f$ is not stable.

\end{teorema}
\begin{remark} {\rm Recall that, for a trigonal canonical curve   $F \subset \mathbb{P}^{g-1}$, ($g\ge 4$), the intersection of all  the quadrics in $\mathbb{P}^{g-1}$ containing $F$ is a rational surface $ \mathbb{F}_n$ embedded with minimal degre $g-1$ in $\mathbb{P}^{g-1}$ (see, e.g.,  \cite{reid}, 2.10, pp. 25). This $n$ is the Maroni invariant of $F$. 
 }
\end{remark}

\begin{remark} {\rm  The base points not being in a general position  means that some of the base points are infinitely near  or,  in cases  ii) and  iv), also three or more in the same line or, in cases  i)  and  v), two or more in the same fiber of the ruling. }
\end{remark}

\begin{remark} {\rm  Theorem \ref{Classification}  results from  combining Propositions  \ref{caseg-2}, \ref{caseg-1}    with Theorem \ref{g17} and Propositions \ref{g11}, \ref{g10}  (see section \ref{end}). }
\end{remark}

\begin{remark} {\rm   For $g\geq 4$, the only known example of a fibration with exactly $5$ singular fibers and non-hyperelliptic general fiber seems to be the one induced by the Wiman-Edge pencil. In this example, $f$ is the resolution of a  plane pencil of degree $6$ and genus $6$ with base points not in a general position (see, for instance, \cite{Dolgachev on WE}, \cite{Edge on Wiman}, \cite{Zamora WE pencil}). In this way, this example shows that there exists at least one fibration satisfying the hypothesis of Theorem \ref{Classification}, iv). }

\end{remark}\bigskip

{\bf Notation and conventions} We work over the complex numbers.  
Given a  nonsingular projective surface $X$  we denote by $K_X$ its canonical divisor and by $e(X)$ its topological Euler characteristic.
Given a divisor $D$ on $X$ we write $h^i(D)$ for $\dim H^i(X, \mathcal{O}_X(D)). $  As usual, $p_g$ denotes the geometric genus and $q$ the irregularity  of $X$.

 A $(-i)$-curve is a nonsingular rational curve with self intersection $-i$.
 
$\mathbb F_n$ is the $\mathbb P^1$-bundle over $\mathbb P^1$ associated to the sheaf $\mathcal O\oplus \mathcal O(-n)$. 

\bigskip 

The paper is organized as follows. In  section  \ref{Preliminaries},  we recall some properties of  fibrations  $f$ on a surface $X$  established in \cite{konno}.

 In section \ref{The MVT Inequality and some consequences},
  we recall the canonical class inequality. This inequality was deduced by Vojta (see \cite{vojta}), as a kind of Miyaoka inequality adapted to fibrations. Later, in \cite{tan}, Tan sharpened the inequality. Thus we call it the MVT (Miyaoka-Vojta-Tan) inequality. From this we derive several consequences, namely that the assumptions $F$ non-hyperelliptic and  $s=5$  imply  $K_X^2= 2-3g$ or $3-3g$. In addition,   we present a more compact proof of the results in \cite{C-Z}. 
  
  In section \ref{four},  we prove two lemmas, that we need in the sequel.
Finally, in section \ref{end},  we lower the bound $g\leq 17$ to $g\leq 11$ and we establish the statement about the base points  and stability of the pencils in Theorem \ref{Classification}.

\section{\bf Preliminaries}\label{Preliminaries}

In this section we recall a few facts that we will need later and that were proven in \cite{konno},  in the course of studying a different problem.

\begin{proposicion}[ \cite{konno}]\label{adjoint1} Let $X$ be a nonsingular projective  surface with $p_g=q=0$ and  $f: X \to \mathbb P^1$ be a relatively minimal fibration  with general   non-hyperelliptic  fiber  $F$ of genus $g\geq 3$. 
  Then:
\begin{itemize}
    
\item[i)] $K_X + F$ is nef;

\item[ii)] $h^0(X,K_X + F)=g$;
\item[iii)] the map  $\varphi$  defined by the linear system $|K_X+F|$ is generically finite;

\item[iv)]  $(K_X+F)^2\geq g-2$; 

\item[v)] the only  curves contracted by $\varphi$ are either $(-2)$-curves contained in fibers  of $f$ or $(-1)$-curves $\theta$ satisfying $\theta\cdot F=1$;

 \item[vi)] if $(K_X+F)^2\leq 2g-5$,  $X$ is a rational surface and $\varphi$ is a birational morphism. 
\end{itemize}
    \end{proposicion}
    
        \begin{proof}

    i) and ii)  are Lemma 1.1 of \cite{konno}.  For  iii), see Lemma 1.3 of \cite{konno}. By ii) and iii),  the image $\Sigma$  of $\varphi$ is a non degenerate surface in $\mathbb P^
{g-1}$,  and so deg $\Sigma\geq g-2$.  Since $(K_X+F)^2\geq $ deg $ \Sigma$ deg  $\varphi$,  we have iv).
Assertion v) is Lemma 1.2 of  \cite{konno}.  Assertion  vi) comes from Lemmas 1.1(2) and 1.4 of \cite{konno}.

\end{proof}

\begin{proposicion}[ \cite{konno}] \label{adjoint2} Let $X$ be a nonsingular projective  surface with $p_g=q=0$ and  $f: X \to \mathbb P^1$ be a relatively minimal fibration  with general hyperelliptic fiber $F$   of genus $g\geq 3$. Then either  $(K_X+F)^2=0$ and the linear system $|K_X+F|$ is composed with a pencil or $(K_X+F)^2\geq 2g-4$.

\end{proposicion}

 \begin{proof} If the linear system $|K_X+F|$ is composed with a pencil, by Lemma 1.3 of \cite{konno}, either $(K_X+F)^2=0$ or $(K_X+F)^2\geq 2g-2$. On the other hand if  $|K_X+F|$ is not composed with a pencil, the degree of the image $\Sigma$ of the map  $\varphi$  defined by the linear system $|K_X+F|$ is $\geq g-2$. Since the general fiber $F$ is hyperelliptic, the degree $d$ of $\varphi$ is $\geq 2$ and therefore we obtain $(K_X+F)^2\geq 2(g-2)=2g-4$. 
 
 \end{proof}

\section{\bf The MVT Inequality and some consequences}\label{The MVT Inequality and some consequences}
Let $f:X\to B$ be a semistable fibration of genus $g\ge 2$. Recall that, if $e_f$ denotes the total number of nodes in the fibers of $f$, then:

$$ e_f= e(X)-4(g-1)(g_B-1) . $$

Denote, as in \cite{ttz}, by $q_1,...,q_r$ the rational double points  obtained after contracting the vertical  connected configurations of $(-2)$-curves via  the morphism $\sigma: X\to X^\#$, where $X^\#$ is the relative canonical model. Since $f$ is semistable, $q_1,...,q_r$  are double points of type $A_{\mu_{q_i}} $, and  $\mu_{q_i}$  is the length of the chain of $(-2)$-curves in $X$ over $q_i$. If $q$ is a singular point of some fiber not contained in any $(-2)$-curve, we write $\mu_q =0$.

\medskip 

\begin{notation} \label{lr} {\rm  In what follows, we will denote by $l'$ the total number  of $(-2)$-curves  contained in fibers of $f $ and by $r$ the total number of double points of $X^\#$.  }  
 \end{notation} 

\medskip

 Set:

$$ r_f= \sum_{q} \frac{1}{1+\mu_q}. $$

\noindent where $q$ is either one of the  rational double points $q_i$ or  a singular point of a fiber lying in the nonsingular part of  $X^\#$. 

Note that $r_f\leq e_f$.  In fact,  since the number of points $q$ such that $\mu_q=0$ is $e_f-\sum_{i=1}^{r} ({1+\mu_{q_i}})$, 
 we have  
 
  \begin{equation}\label{rfef}
  r_f=e_f-\sum_{i=1}^{r}(1+\mu_{q_i})+\sum_{i=1}^{r}\frac{1}{1+\mu_{q_i}}.\end{equation}
  
With    $l'$  and $r$ as in notation \ref{lr}, one has    $$l'=\sum_{i=1}^{r}\mu_{q_i}.$$  
So, in particular,
 
 \begin{equation}\label{eq-ttz} 
 r_f\leq e_f-l'-\frac{r}{2} \end{equation}
 
The above  is the content of Lemma 1.1 of  \cite{ttz}.
\medskip

Let $s$ be the number of singular fibers of the fibration $f$.  For any integer $e\ge 2 $, Sheng-Li Tan, using results in \cite{vojta},  proved the following (see inequality (5) in the proof of Lemma 1.3 of \cite{tan}):

\begin{equation}\label{MVT}\tag{MVT} e^2 (K_X^2- (2g-2)(6 g_B-6+s-s/e))\le 3 r_f \le 3e_f.\end{equation}

For  $B= \mathbb{P}^1$  the (\ref {MVT}) inequality takes the form 

\begin{equation}\label{MVT2} e^2 (K_X^2- (2g-2)(-6+s-s/e))\le 3 r_f \le 3e_f,\end{equation}

When $X$ is rational, then $B$ is necessarily  $\mathbb{P}^1$. If moreover, we assume $s=5$, then the inequality takes the form:
 
\begin{equation*}
    \label{MVT5}\tag{MVT 5} e^2(K_X^2+(2g-2) (1+\frac{5}{e}))\le 3 r_f \le 3e_f. \end{equation*}

We obtain from  (\ref{MVT5}) the following results, that were  proven  in \cite{C-Z}, and,   according to  \cite{ConjectureTan},  also  partly proven in the thesis in chinese   \cite{Yu}.
The proof presented here is a simplified and unified version of the proofs appearing in \cite{C-Z}.

\begin{teorema}\label{g17} Let $X$ be a  rational nonsingular projective surface  and  $f: X \to \mathbb{P}^1$ be a semistable  non isotrivial fibration of genus $g\ge 4$,  having exactly 5 singular fibers.  Then either: 

\begin{itemize}
    
\item[i)]  the general fiber $F$ is hyperelliptic,  $(K_X+F)^2=0$ and $|K_X+F|$ is composed with a pencil 
or

\item[ii)]  the general  fiber $F$ is  non-hyperelliptic, and either $K_X^2=2-3g$, $h^0(2K_X+F)=0$ and $g\leq 17$ or $K_X^2=3-3g$, $h^0( 2K_X+F)=1$ and $g\leq 10$.  

\end{itemize}
 
\end{teorema}

\begin{proof}

Note that since $X$ is rational $e(X)= 12-K_X^2$ and thus $$e_f=4(g-1)+12-K_X^2.$$

Evaluating (\ref{MVT5}) at $e=5$, we obtain:

$$K_X^2+4(g-1)\le \frac{3r_f}{25},$$
or, equivalently, setting  $\alpha:=e_f-r_f$:

$$K_X^2+4(g-1)\le \frac{12(g-1) +36-3K_X^2 -3\alpha}{25},$$

\noindent i.e., equivalently,

\begin{equation}
    \label{G5}7K_X^2 +22(g-1)\le 9 - \frac{3\alpha}{4}.\end{equation}
    
    \medskip
 Assume now that $(K_X+F)^2>0$.  Then, by Proposition \ref{adjoint1}, iv) and Proposition \ref{adjoint2},  $(K_X+F)^2\geq g-2$, i.e., $K_X^2\ge 2-3g$.
Set  $K_X^2=2-3g+a$ (with $a\geq 0$).   From (\ref{G5}) and the assumption $g\geq 4$, we obtain$$4\leq g \leq 17-(3\alpha/4)-7a.$$ 

Hence $a=0$ or $a=1$, i.e., $K_X^2=2-3g$ or $K_X^2=3-3g$, yielding, in the first case, $g\leq 17$ and, in the second case, $g\leq 10$.

Note that, since we are assuming $g\geq 4$, we have $(K_X+F)^2<2g-4$  and so, by Proposition \ref{adjoint2},   the general curve $F$ is non-hyperelliptic.    

So, by  Proposition \ref{adjoint1}, i), $K_X+F$ is nef.  Since we are assuming that $(K_X+F)^2>0$, the Kawamata-Viehweg vanishing theorem yields $h^1(2K_X+F)=0$.  From the Riemann-Roch  theorem  we obtain

$$h^0(2K_X+F)= K_X^2 +3g -2.$$

So, if $(K_X+F)^2>0$, we have the possibilities stated in ii) and in particular  the general curve $F$ is non-hyperelliptic.

    \medskip

If $(K_X+F)^2=0$, then,  by Proposition \ref{adjoint1},   the general curve $F$ must be  hyperelliptic and, by Proposition \ref{adjoint2},  $|K_X+F|$ is composed with a pencil, proving thus i).\end{proof}

The following Lemmas will be useful in the proof of Theorem \ref{Classification}.

\begin{lema}\label{l lt 4} Let $X$ be a  rational nonsingular  projective surface  and  $f: X \to \mathbb{P}^1$ be a semistable non isotrivial  fibration of genus $g\ge 4$  having exactly 5 singular fibers.
If  $K_X^2=2-3g$, then the number $l'$ of $(-2)$-curves contained in the fibers of $f$ satisfies

$$ 4(g-1)+3l'+\frac{3r}{2}
\leq 64,$$
where $r$ is the number of points  $q_i$ such that $\mu_{q_i}\neq 0$.
\end{lema}

\begin{proof}
From (\ref{MVT5}) evaluated in $e=5$, we obtain:
\begin{equation}\label{mvt}
    \frac{25}{3}(g-2)\leq r_f\leq e_f=7(g-1)+13.
\end{equation}

By  (\ref{eq-ttz}),  we have 
\begin{equation}\label{eq-ttz2}
    r_f+l'+\frac{r}{2}\leq e_f.\end{equation}
 The statement follows by combining (\ref{mvt}) and (\ref{eq-ttz2}).

\end{proof}

\begin{lema}\label{stable}  Let $X$ be a  rational nonsingular  projective surface  and  $f: X \to \mathbb{P}^1$ be a  semistable non isotrivial  fibration of genus $g\ge 4$  having exactly 5 singular fibers. If  $K_X^2=3-3g$, then: 

\begin{itemize}
\item[i)]  if $g\geq 9$,  $f$ is stable;
\item[ii)]  if $g=8$, the relative canonical model $X^\#$ of $X$  has at most one singular point  of type $A_1$ or $A_2$.
\end{itemize}
\end{lema}

\begin{proof}
From (\ref{MVT5}) evaluated in $e=5$, we obtain:
\begin{equation}\label{des-voj}
    \frac{25}{3}(g-1)\leq r_f\leq e_f=7(g-1)+12.\end{equation}

Let  $l'$ and $r$ be as in Notation \ref{lr}. Combining (\ref{des-voj}) and (\ref{eq-ttz}) we have:
\begin{equation}\label{eq-star}
    4(g-1)+3l'+\frac{3r}{2}\leq 36.
\end{equation}
Using $g-1\geq 8$ in (\ref{eq-star}), we obtain $3l'+\frac{3r}{2}\leq 4$ and so $l'=r=0$, i.e., $f$ is stable.

If  $g-1=7$,  (\ref{eq-star}) gives $3l'+\frac{3r}{2}\leq 8$.  So,  if $l'>0$,   necessarily $r=1$ and $l'\leq 2$, proving ii).

\end{proof}

\section{\bf Two Lemmas}\label{four}

In this section   $f:X \to \mathbb{P}^1$ will be a semistable, non isotrivial  fibration with  general  non-hyperelliptic fiber $F$ of genus $g\ge 4$ on the rational surface $X$, having exactly $s=5$ singular fibers.

  From Theorem \ref{g17}, we know that  either $K_X^2=2-3g$ and $g\leq 17$ or $K_X^2=3-3g$ and $g\leq 10$. \bigskip 

Denote by $F_1,\dots, F_5$ the $5$ singular fibers. Write:

$$F_i = \sum_j F_{ij}$$
for the decomposition into irreducible components, $p_{ij}$ for 
the arithmetic genus
of $F_{ij}$ and $g_{ij}$ for its geometric genus. For each $i$, we have:
\begin{align} 
    2(g - 1) &\nonumber = K_X.(\sum_j F_{ij})\\ &\nonumber = \sum_j 2(p_{ij} -1) - \sum_j F_{ij}^2 \\ & \nonumber= \sum_j 2(g_{ij} - 1) + 2n_i, \end{align}
with $n_i$ denoting the number of nodes in the fiber $F_i$.
Summing on $i$, we obtain:
\begin{equation}\label{sum}
    5(g - 1) = \sum_{ij} (g_{ij} - 1) + e_f. 
\end{equation}

\begin{lema}\label{gij} With the hypothesis above:
\begin{itemize}
    \item [i)] if $K_X^2 =2-3g$, then $$2g+11=-\sum_{ij} (g_{ij} - 1); $$
    \item[ii)] if $K_X^2=3-3g$, then $$2g+10=-\sum_{ij} (g_{ij} - 1).$$
\end{itemize}
    \end{lema}
    
    \begin{proof} It follows at once from (\ref{sum}) and the fact that
    $$e_f= 4(g-1)+e(X)=4g+8-K_X^2.$$
    \end{proof}

We will need also the following:

   \begin{lema}\label{DoubleFibration}  Assume that $X$ admits also a fibration $\gamma: X\to \mathbb{P}^1$ of genus $0$   with general fiber $\Gamma$  such that $F\cdot\Gamma=m$. Then:

    \begin{itemize} 
        \item [i)]  any component of a fiber of $\gamma$ is a nonsingular rational curve with self-intersection $\geq -m$;
        \item [ii)]  a fiber of $\gamma$ contains at most one curve $\theta$ with $\theta^2=-m$  and  if such a curve exists it is also a component of a fiber of $f$;
        
    \item [iii)]  a reducible fiber of $f$ has at most $m$ components not contained in fibers of $\gamma$;

\item  [iv)] $$-\sum_{ij} (g_{ij} - 1)\le 5m+c,$$ where $c$ denote the number of irreducible curves common to fibers of $f$ and $\gamma$.  
\end{itemize}
 \end{lema}

 \begin{proof}

 i) ii)   Obviously  every component of a fiber of $|\Gamma|$ is a nonsingular rational curve.  Note that $\Gamma\cdot(K_X + F) = m-2$ and so,  since $K_X+F$ is nef,  $(K_X+F)\cdot\theta\leq m-2$, for any component  $\theta$ of a fiber of $\gamma$. Thus, from the adjunction formula, $\theta^2\geq -m$.  If $\theta^2=-m$, then $K_X\cdot\theta=m-2$   implies    that $F\cdot\theta=0$  and   $(K_X+F)\cdot\theta= m-2$.    So, necessarily, this fiber of $\gamma$ does not contain other curve with self-intersection $-m$  (in fact, any other component of this fiber of $\gamma$ must be  contracted by the  map $\varphi$ defined by $|K_X+F|$)  and from $F\cdot \theta=0$ we conclude that $\theta$ is a component of some fiber of $f$.   
 \smallskip 
 
  Since $|\Gamma|$ is nef  and $F.\Gamma=m$,  assertion iii) is clear.  
  
   \smallskip 
   
iv) follows from the assumption $s=5$ and iii).
 \end{proof}

\begin{notation} \label{l} {\rm  In what follows, we will denote by $l_i$ the number  of $(-i)$-curves common to fibers of $f$ and of a fibration $\gamma$ as in Lemma \ref {DoubleFibration}.   }  
 \end{notation} 

\section {\bf The proof of Theorem \ref{Classification}}\label{end}

Recall  that, by Theorem \ref{g17},  either $K_X^2=2-3g$ and $g\leq 17$ or $K_X^2=3-3g$ and $g\leq 10$.
We start by discussing the first case. 

Note that $K_X^2=2-3g$ means $(K_X+F)^2=g-2$.
\medskip

 The results in the following Proposition are  implicit in  \cite{saito}  (see Theorem 4.1 and the proof of Lemma 3.1  of loc. cit.).

\medskip

\begin{proposicion}\label{caseg-2}  Let $f:X \to \mathbb P ^1$ be a
 relatively minimal
fibration on the  rational  nonsingular projective surface $X$ with general   non-hyperelliptic  fiber of genus 
 $g\ge 4$ such that $(K+F)^2=g-2$. 
Let $\xymatrix{\varphi:X\ar@{-->}[r] &  \Sigma\subset\mathbb{P}^{g-1}}$
be the map defined by $|K_X+F|$. Then $\varphi$ is a birational morphism and  either: \begin{itemize}
	\item [i)] $g=6$, $\Sigma$ is the Veronese surface in $\mathbb P^5$ and  the fibration $f$ is  obtained blowing up  the base points  of a pencil  with general nonsingular member  of plane quintics   or
	
	\item [ii)]  $\Sigma$ is $ \mathbb{F}_n$  with $3n\leq g+2$, embedded in   $\mathbb P^{g-1}$  by the linear system    $|\Delta+  (\frac{g+n}{2}-1) \Gamma|$, where  $\Gamma$ is a fiber of the ruling of $\mathbb{F}_n$ and $\Delta$ is the section such that $\Delta^2=-n$.   The fibration $f$ is  obtained blowing up   the  base points  of a pencil  of curves with general nonsingular member  in  $|3\Delta+(\frac{g+n}{2}+n+1)\Gamma|$, or 
	
	 \item [iii)]  $g=4$,  $\Sigma$ is  the quadric cone in $\mathbb P^3$ and the fibration $f$   is  obtained blowing up   the  base points  of a pencil  with general nonsingular member  of  cubic hypersurface sections of $\Sigma$.	\end{itemize}
	
	In all cases, the base points of the pencils corresponding to the fibration are simple base points (but possibly infinitely near).
\end{proposicion}
\begin{proof}  
 Since $(K_X+F)^2=g-2$,   by Proposition \ref{adjoint1},vi), the map defined by the linear system $|K_X+F|$ is a birational morphism.

Then    $\Sigma$ is a surface of  degree $g-2$ in $\mathbb{P}^{g-1}$, and so, by  \cite{Nagata}, (cf., e.g., \cite{Dolgachev}),  $\Sigma$ is either: a rational normal scroll or the Veronese surface $\iota_{2H}(\mathbb{P}^2)\subset\mathbb{P}^5$, or a cone on a rational normal curve of degree $g-2$.

If $\Sigma$ is the Veronese surface, then 
$$\iota_*\O_{\mathbb{P}^2}(2)=\O_{\Sigma}(1)=\varphi_*(K_X+F)=\iota_*\O_{\mathbb{P}^2}(-3)+F',$$
so the fibration is obtained by   blowing up  the base points of a pencil  $\Lambda$  of plane quintics on $\mathbb{P}^2.$   Since $g=6$  the general member of this pencil must be nonsingular and thus the base points are simple (possibly infinitely near).

If $\Sigma$ is a rational normal  scroll, there exists $\xymatrix{\Psi:\mathbb{F}_n\ar@{^(->}[r]^{|\Delta+k\Gamma|} &  \Sigma\subset\mathbb{P}^{g-1}}$ for some $n>0$ and $k>n$, where $\Delta$ is the section such that $\Delta^2=-n$,   $\Gamma$ is a fibre of the structural morphism of $\mathbb{F}_n$ and

   $$ g+n=2(k+1).$$

On the other hand, taking into account that $K_{\mathbb{F}_n}=-2\Delta-(n+2)\Gamma$ and that $\Psi$ is a birational morphism, we obtain
$$\Psi_*(\Delta+k\Gamma)=\O_{\Sigma}(1)=\varphi_*(K_X+F)=\Psi_*(-2\Delta-(n+2)\Gamma)+F'.$$

Then $F'\equiv 3\Delta +(k+n+2)\Gamma$ and $F'.\Gamma =3$.  So the fibration is the blow up of the base points of a pencil of curves  $\Lambda$  in $|3\Delta +(k+n+2)\Gamma|$.  Since a general curve in this linear system has genus $g$,  the general member of the pencil must be nonsingular and thus the base points are simple (possibly infinitely near).

Since $F'$ is irreducible, $\frac{g+n}{2}+n+1\geq 3n$, so $g+2\geq 3n.$

Finally, assume that $\Sigma $  is a cone on a rational normal curve.  Since, by Proposition \ref{adjoint1},  the only curves contracted by $\varphi$ are $(-1)$ or $(-2)$-curves, the vertex of the cone must be a $(-2)$-curve. So this possibility can only occur if $g=4$.   In that case, it is easy to see that we have case iii).  
\end{proof}

Finally we are in a position to prove the following:
\begin{proposicion}\label{g11} Let $X$ be a  rational nonsingular  projective surface  and  $f: X \to \mathbb{P}^1$ be a semistable non isotrivial  fibration of genus $g\ge 4$  having exactly 5 singular fibers.  Then, if $K_X^2=2-3g$,     $g\leq 11$. 

Furthermore, if the image $\Sigma$ of $X$ by the morphism $\varphi$ defined by $|K_X+F|$ is $\mathbb F_n$ or the quadric cone, the fibration $f$ is not strongly stable and, if $g\geq 7$, it is also not stable. Finally,  the base points of the pencil  $\Lambda$  in $\Sigma$  corresponding to the fibration $f$  are not in a general position.
\end{proposicion}

\begin{proof} 1) Suppose that the image of $\varphi$  is $\mathbb F_n$, as in  ii) of Proposition \ref{caseg-2}.
 
Abusing notation, we denote also by $\Gamma$ the general fiber of the fibration $\gamma$ in $X$ coming from the ruling of  $\mathbb{F}_n$.
Then  $F.\Gamma=3$. Combining Lemmas \ref{gij}, i) and \ref{DoubleFibration}, iv), we obtain:

\begin{equation}\nonumber\label{bound-2-3curves}
    2g+11= -\sum_{ij} (g_{ij} - 1)\le 15 +c.
\end{equation}

Note that, being $f$ relatively minimal, by Lemma \ref{DoubleFibration}, i), every curve common to fibers of  $f$ and $\gamma$ must be either a $(-2)$ or a $(-3)$ $f$-vertical curve. Let $l_i$ be as in Notation \ref{l}.

A $(-3)$-curve contained in a fiber of $f$ and a fiber  of $\gamma$  satisfies $(K_X+F)\cdot \theta=1$. As such, its image via $\varphi$ is a ruling of $\mathbb F_n$. Note that   one needs to blow up three times this ruling to obtain $\theta$.  

By Lemma \ref{DoubleFibration}, ii), each fiber of $\gamma$ contains at most one $(-3)$-curve. Since $K_X^2=2-3g$,   the number of blown up points  in the fibers of  $\gamma$  is  $3g+6$.    Thus  $l_3 \le g+2$. Therefore:

$$2g  + 11 = - \sum_{ij} (g_{ij}-1)\le 15+g+2 +  l_2.$$
Thus, we obtain:

\begin{equation}\label{6+l} g\le 6+ l_2 \le 6+ l', \end{equation}
($l'$ denoting, as in the previous section the total number of $f-$vertical $(-2)$-curves).
Or, equivalently, 

$$3g-18 \le 3l'.$$

Substituting  this last inequality in the inequality of Lemma \ref{l lt 4}, we obtain $g\le 12$. 

 Now, if $g=12$ we have by (\ref{6+l}) that $6\le l'$.  By  Lemma \ref{l lt 4}, we conclude that $l'=6$ and $r=1$, i.e., there is a unique chain of vertical $(-2)$-curves of length exactly $6$.

 Thus, by equation (\ref{rfef}) in Section 3, 

 \begin{align*}
     r_f &= e_f-7+\frac{1}{7} \\ &= 7g+6 -7+\frac{1}{7}\\
 &= 83 +    \frac{1}{7}.
 \end{align*}

 Substituting this value in (\ref{MVT5}) evaluated in $e=5$, we obtain a  contradiction. Hence $g\leq 11$. 
 
 \smallskip

Finally, note that,  from what we have seen above, for $g\geq 7$,  one has  $l_2>0$ and therefore the fibration $f$ is not stable.  A $(-2)$-curve common to a fiber of $f$ and of $\gamma$ is contracted by $\varphi$. Since $\varphi$ only contracts the exceptional divisors corresponding to the base points of  the pencil $ \Lambda$,   these $(-2)$-curves  common to a fiber of $f$ and of $\gamma$  are components of exceptional divisors, and so  the pencil $\Lambda$ has infinitely near base points.

In addition, since $l_2+l_3=0$ implies  $2g+11\le 15$  and we are assuming $g\geq 4$,  we see that $f$ is never strongly stable.   Hence, for all $g$ between $4$ and $11$, the base points of the fibration are not in general position in the sense that either the fibration  has infinitely near base points or it has three base points lying in the same fiber of the ruling of $\Sigma$.
\medskip

2) If $\Sigma$ is the quadric cone then $g=4$  and $2g+11=19$. Considering  the ruling of $\Sigma$ by $\Gamma$,  we obtain again  $l_2+l_3\geq 0$ and, as above, we see that $f$ is never strongly stable and the base points of the pencil giving rise to the fibration $f$ are not in a general position in the same sense as above. 

\medskip

3) If $\Sigma $ is the Veronese surface, then  $g=6$  and on $X$ we have genus $0$ fibrations determined by the linear system  corresponding to the lines in $\mathbb P^2$ through a base point  of the pencil $\Lambda$ in $\mathbb P^2$ corresponding to $f$. Fix one of these and denote it by $\gamma$. In this case $F.\Gamma=4$ and using Lemmas \ref{gij}
  and \ref{DoubleFibration} we obtain:

  $$ 22 \le 20 +l_2+l_3+l_4.$$
  
  Now remark that $l_3$ must be 0. In fact, a $(-3)$-curve $\theta$ contained in a fiber of $f$ satisfies $(K_X+F)\cdot\theta=1$, and as such  it is mapped  to a line via $\varphi$,   whilst the Veronese surface contains no lines.  
  
  So $l_2+l_4\geq 2$. As before, $l_2\neq 0$ means that the pencil $\Lambda$  of quintic curves in $\mathbb P^2$ giving raise to $f$ has infinitely near base points. On the other hand,  $l_4\neq 0$  means that one of the quintics contains a line and  that on this line there are 5 base points of $\Lambda$ (maybe infinitely near). So, anyway, the base points  are not in general position  (5 may be on the same line  in $\mathbb P^2$ or some  may be infinitely near).

\end{proof}

Now we turn to the case $K_X^2=3-3g$.  Recall that  $K_X^2=3-3g$ means $(K_X+F)^2=g-1$. 

\medskip
The contents of the  next Proposition are already in \cite {konno}, as explained in the following  proof.

\medskip

\begin{proposicion}\cite{konno}\label{caseg-1}  Let $f:X \to \mathbb P ^1$ be a
 relatively minimal
fibration on the  rational nonsingular projective surface $X$ with general   non-hyperelliptic fiber of genus 
 $g\ge 4$ such that $(K+F)^2=g-1$. 
Let $\xymatrix{\varphi:X\ar@{-->}[r] &  \Sigma\subset\mathbb{P}^{g-1}}$
be the map defined by $|K_X+F|$. Then $\varphi$ is a birational morphism.

 If $\rho: Y\to \Sigma$ is  the minimal desingularization $Y$ of $\Sigma$, then   $\Sigma$ is the image of $Y$ by the morphism defined by $|-K_Y|$. 
 
  Furthermore, $Y$ is a (possibly weak) Del Pezzo surface, i.e.,   
 
 \begin{itemize} 
 
  \item   [i)] $Y$ is  $\mathbb P^2$ blown up in $10-g$ points in weakly general position  or
  
 \item   [ii)] $g=9$,  $\Sigma$ is  the Veronese image of a quadric in $\mathbb P^3$ ($Y=\mathbb F_0$ or $Y=\mathbb F_2$).  \end{itemize}

In addition,    $g\leq 10$ and the fibration is obtained by blowing up the simple  (but possibly infinitely near) base points of a pencil of curves with general  nonsingular member in the linear system $|-2K_Y|$, i.e.,  in case i), of a pencil of of plane degree $6$ curves, whose general element admits only  $10-g$  singularities of order $2$,
and, in case ii),  a pencil    with general nonsingular member  of quartic hypersurface sections  of a quadric in $\mathbb P^3$. 

\end{proposicion}
\begin{proof}  Since $(K_X+F)^2=g-1$,   by Proposition \ref{adjoint1}, vi), the map defined by the linear system $|K_X+F|$ is a birational morphism. 

Since $|K_X+F|$ is a complete linear system, by the classification of surfaces of degree $d$ in $\mathbb P^d$ in \cite{Nagata}, $\Sigma$ and its minimal resolution $Y$ are as stated and $g\leq 10$  (cf. also, e.g., \cite{Dolgachev}, Chp. 8, also the  proof of  (1)  of Proposition 2.2 of  \cite {konno}) .

We have  a morphism $\phi:X\to Y$ such that    $\varphi= \rho\circ \phi$. 
 The proofs of  (1)  of Proposition 2.2   and  of  Lemma 1.2 of \cite {konno} show  that, under our hypothesis,     the fibration $f$ is obtained  by blowing up the simple (but possibly infinitely near) base points of a pencil  of curves  $|G|$  in $|-2K_Y|$, such that $\phi_*(F)=G$ (and $\phi^*(K_Y+G)=K_X+F)$).   
 
If $\Sigma$ is not the singular quadric in $\mathbb P^3$, we have the description in the statement. If $\Sigma$ is the singular quadric in $\mathbb P^3$, it suffices to notice that $|-K_Y|=\rho^*|4H|$ where $H$ is an hyperplane of $\mathbb P^3$.

\end{proof}

 \bigskip
 
 \begin{proposicion}\label{g10} Let $X$ be a  rational nonsingular  projective surface  and  $f: X \to \mathbb{P}^1$ be a semistable non isotrivial  fibration of genus $g\ge 4$  having exactly 5 singular fibers.  Then, if $K_X^2=3-3g$ and $g\geq 6$,  the base points of the pencil  of sextic curves in $\mathbb P^2$  or, in the case of the quadric,   of the pencil in $|4H|$  giving rise to the fibration $f$ are never in a general position.  Furthermore the quadric is nonsingular.
 
 \end{proposicion} 
 
  \begin{proof}  
   If  $g\leq 9$,  we  find on $X$ a fibration  $\gamma$, as in Lemma \ref{DoubleFibration},  such that $\Gamma\cdot F=4$, by taking, in the case of the quadric $Q$, the pencil corresponding to a ruling of $Q$ 
 and, in the case of the  image of $\mathbb P^2$, the pencil corresponding to the lines  in $\mathbb P^2$ passing through one of the double points of the pencil of  sextic curves.  We will, as before, abuse notation and also denote by $\gamma$ the corresponding fibration on $X$.

 As before, combining Lemmas \ref{gij}, i) and \ref{DoubleFibration}, iii), we obtain:
$$ 2g +10= - \sum_{ij} (g_{ij} - 1)\le 20 +c.$$

We have  then:
\begin{equation}\label{c}
   l_2+l_3+l_4\geq 2g-10.
\end{equation}
\medskip 

So for  $g\geq 6$,   $l_2+l_3+l_4>0$. Note that  $l_2+l_3+l_4>0$ implies that the base points of the pencil of sextics giving rise to the fibration are not in general position (either infinitely near or more than 3 on a line).

  \medskip 
  
 For $g=9$,  remark first that, by Lemma \ref{stable},    $X$ has no $(-2)$-curves.  So  the quadric is nonsingular and $l_2=0$, implying  $l_3+l_4\geq 8$. 
 
  Since a $(-3)$-curve contained in a fiber of $f$  satisfies $(K_X+F)\cdot \theta=1$,  its image via $\varphi$ is a line.  The Veronese embedding in $\mathbb P^8$ of a quadric contains no lines. On the other hand,  the image in $\mathbb P^8$ of $\mathbb P^2$ blown up in one point $P$ contains only the line corresponding to the exceptional divisor. So, by our choice of $\gamma$, any $(-3)$-curve contained in a fiber of $f$ is not a component of a fiber of $\gamma$ in $X$.
  
  So $l_3$ must be $0$.  Thus,  by  (\ref{c}),  $l_4\geq 8$.  Note that, to obtain  from a ruling $r$ of $Q$  a $( -4)$-curve, we need to do  blow up four points on $r$, and to obtain from a line $l$ in the plane  a $(-4)$-curve, we need to blow up $5$ points on $l$.
 
 \medskip
 
If $g=10$,  take for $\gamma$  the pencil coming from  the lines in  $\mathbb P^2$ through one of the base points of the pencil of sextics. In that case, $\Gamma\cdot F=5$ and we obtain 
 
 $$2g+10=30= -\sum_{ij} (g_{ij} - 1)\le 25 +c.$$

Hence $l_2+l_3+l_4+l_5\geq 5$. By  Lemma \ref{stable},  $l_2=0$ and, since the Del Pezzo surface of degree 9 does not contain either lines or conics, we have $l_3=l_4=0$. Note that to obtain a $(-5)$-curve we need to blow-up 6 points  on a line.

 \end{proof}

\medskip
\begin{remark} {\rm   In the case $g=10$ and $K^2=3-3g$, it is not difficult to see with the techniques above that all the sextics giving rise to singular fibers would be completely reducible (i.e., the union of  lines). By  \cite{pereira}, a fibration can have at most five completely reducible fibers and there is no known example attaining this bound.  
We were not able to prove or disprove the existence of this example. }

\end{remark}

\vskip2truecm

\begin{minipage}{13.0cm}
\parbox[t]{6.5cm}{Margarita Casta\~ neda-Salazar\\
Instituto Educativo Ammadeus A.C.\\
Vialidad Vetagrande 468. Zona conurbada \\
C.P. 98600. Guadalupe, Zacatecas\\
 Mexico\\
mateliceo0@gmail.com
 } \hfill
\parbox[t]{5.5cm}{Margarida Mendes Lopes\\
Departamento de  Matem\'atica\\
Instituto Superior T\'ecnico\\
Universidade de Lisboa\\
Av.~Rovisco Pais\\
1049-001 Lisboa, Portugal\\
mmendeslopes@tecnico.ulisboa.pt\\
orcid: 0000-0002-8921-291X
}

\vskip1.0truecm

\parbox[t]{7.0cm}{Alexis Zamora\\
Universidad Aut\'onoma de Zacatecas\\
Unidad Acad\'emica de Matem\'aticas\\
Av. Solidaridad y Paseo de la Bufa\\
C.P. 98000, Zacatecas, Zac\\
 Mexico\\
 alexiszamora@uaz.edu.mx}
\end{minipage}


\begin{thebibliography}{99}

\bibitem{AlCaZa} C. R. Alc\'antara, A. Castorena and A. G.
Zamora, {\em On the slope of relatively minimal fibrations on
complex rational surfaces}. Collect. Math., \textbf{62},   1--15, 2011.



\bibitem{beau-82 A}
Beauville, A.
{\em Le nombre minimum de fibres singuli\'eres d'une courbe stable sur $\mathbb{P}^1$ .}
Ast\'erisque \textbf{86}, 97--108, 1981.


\bibitem{beau-82}
Beauville, A.
{\em Les families stables de courbes elliptiques sur $\mathbb{P}^1$ admettant quatre fibres singuli\'eres.}
C.R. Acad. Sc. Paris,  \textbf{294},  657--660, 1982.




\bibitem{C-Z}
Casta\~ neda, M., Zamora, A. G.
{\em Semistable fibrations over $\mathbb{P}^1$ with five singular fibers.}
Bol. Soc. Mat. Mex. \textbf{25}, 13--19, 2019. 
\bibitem{Dolgachev} Dolgachev, I. {\em Classical algebraic geometry. A modern view.} Cambridge: Cambridge University Press, xii, 639 pp., 2012.

\bibitem{Dolgachev on WE} Dolgachev, I., Farb, B. \&  Looijenga, E. {\em Geometry of the Wiman-Edge pencil, I: algebro-geometric aspects.}
European Journal of Mathematics
\textbf{4}, 879--930, 2018.

\bibitem{Edge on Wiman} Edge, W. L. {\em A pencil of four-nodal plane sextics.} Math. Proc. Cambridge Philos. Soc. \textbf(89), 413--421, 1981.



\bibitem{ConjectureTan} Gui-Min, Li,  Wan-Yuan, Xu,  Hui-Ting, Zhang
{\em A remark on a conjecture of Tan.} Monatshefte f\"ur Mathematik, https://doi.org/10.1007/s00605-023-01841-2, 2023.

\bibitem{saito} Khac, N.V., Saito, M.-H.  {\em On Mordell-Weil lattices for non-hyperelliptic fibrations on surfaces with zero geometric genus and irregularity.}  Izv. Math. \textbf{66}, No. 4, 789-805 (2002); translation from Izv. Ross. Akad. Nauk, Ser. Mat. \textbf{66}, No. 4, 137-154 (2002).

\bibitem{konno} Kitagawa, S. , Konno, K. {\em Fibred rational surfaces with extremal Mordell-Weil lattices.}  Math. Z.  \textbf{251}, No. 1, 179-204 (2005).

\bibitem{ltz}
Lu, X., Tan, S.-L., Zuo, K.
{\em Singular fibers and Kodaira Dimensions.}
Math. Ann, \textbf{370},  1717--1728, 2018.

\bibitem{Nagata} Nagata, M.
{\em On rational surfaces I.} Memoirs of the College of Science, University of Kyoto, Series A, Vol XXXII, 351--370, 1960.

\bibitem{pereira}
Pereira, J. V.,  Yuzvinsky, S.
{\em Completely reducible hypersurfaces in a pencil.}
Adv. Math. \textbf{219}, No. 2, 672--688, 2008.

 \bibitem{reid} Reid, M.: {\em Chapters on algebraic surfaces.} In: Koll\'ar, J. (ed.) Complex Algebraic Geometry, IAS/Park City Math. Ser. 3, Amer. Math. Soc., 3--159. Providence, RI, (1997) 
 
\bibitem{ttz}
Tan, S.-L., Tu, Y., Zamora, A. G.
{\em On complex surfaces with 5 or 6 semistable singular fibers over $\mathbb{P}^1$.}
Math. Z. \textbf{249},  427--438, 2005.

\bibitem{tan}
Tan, S.-L.
{\em The minimal number of singular fibers of a semistable curve over $\mathbb{P}^1$.}
J. Algebraic Geom. \textbf{4}, 
591--596, 1995

\bibitem{vojta} Vojta P.
{\em Diophantine Inequalities and Arakelov Theory. } Appendix to
introduction to Arakelov theory by S. Lan, Springer-Verlag,
155--178, 1988.



\bibitem{Yu} Yu, F. {\em On the minimal number of singular fibers of semi-stable pencils (In Chinese)}, Ph.D. Thesis, East
China Normal University, 2008.

\bibitem{Zamora WE pencil} Zamora, A. G. {\em Some remarks on the Wiman-Edge pencil.} Proceedings of the Edinburgh Mathematical Society \textbf{61},  401---412,  2018.




\end{thebibliography}
\end{document}